\documentclass[11pt]{amsart}
\usepackage[margin=3cm]{geometry}
\usepackage{bm}
\usepackage{amsthm}
\usepackage{amssymb}
\usepackage{mathrsfs}
\usepackage{hyperref,xcolor,graphicx}
\usepackage[foot]{amsaddr}

\newcommand{\R}{\mathbb{R}}
\newcommand{\Z}{\mathbb{Z}}
\newcommand{\e}{\epsilon}

\DeclareMathOperator{\supp}{supp}
\DeclareMathOperator{\ind}{ind}
\DeclareMathOperator{\crit}{crit}

\begin{document}

\begin{abstract}
In this short note, using OpenAI's resolution of the Solomon--Yau conjecture, we extend our joint work \cite{ChenGaspar23,ChenGaspar25} with Pedro Gaspar to all dimensions $n\geq2$. We construct eternal Allen--Cahn flows on $S^{n+1}$ connecting approximations of any prescribed least-area Clifford hypersurface and equatorial sphere. We also construct infinitely many eternal integral Brakke flows from least-area Clifford hypersurfaces to equatorial spheres.
\end{abstract}

\title[Flows from Clifford hypersurfaces to equators]{Allen--Cahn and mean curvature flows from Clifford hypersurfaces to equators}
\author{Jingwen Chen}
\address{Faculty of Science
University of Macau,\newline
E11, Avenida da Universidade, Taipa,
Macau}
\email{jingwenchen@um.edu.mo}
\maketitle

\newtheorem{thm}{Theorem}
\newtheorem*{thm*}{Theorem}
\newtheorem{lem}{Lemma}
\newtheorem{prop}{Proposition}
\newtheorem{coro}{Corollary}

\theoremstyle{definition}
\newtheorem*{defi}{Definition}
\newtheorem*{ex}{Example}
\newtheorem*{rmk}{Remark}

\section{Introduction}

In our previous papers \cite{ChenGaspar23,ChenGaspar25} joint with P. Gaspar, we studied low energy solutions of the \emph{parabolic Allen-Cahn equation}
\begin{equation} \label{PAC}
\e \, \partial_t u = \epsilon \Delta_g u - \frac{1}{\epsilon} W'(u)\tag{PAC}
\end{equation}
in the $3$-sphere in connection with the \emph{mean curvature flow}. Here $W$ is a nonnegative double-well potential with two wells at $\pm 1$, such as $W(u)=(1-u^2)^2/4$. Our previous results are the following:

\begin{thm*} [Chen-Gaspar, \cite{ChenGaspar23}]
There exist (infinitely many) eternal weak mean curvature flows $\{\Sigma_t\}$ in $S^3$ connecting a Clifford torus ($S^1(1/\sqrt{2})\times S^1(1/\sqrt{2}) \subset S^3$) to equatorial spheres $S^2$. These flows are smooth for large $|t|$.  
\end{thm*}

Let $v_\e^{-\infty}$ be a critical point of $E_\e$ in $S^3$ which vanishes precisely along $\{x \in S^3 \mid x_1^2 + x_2^2 = x_3^2 + x_4^2 = 1/2\}$. 

\begin{thm*} [Chen-Gaspar, \cite{ChenGaspar25}]
For sufficiently small $\e>0$, for any ground state solution $v_\e^{+\infty}$ of the Allen-Cahn equation, there is an eternal solution $v_\e \colon S^3 \times \R \to \R$ to \eqref{PAC} such that $\| v_\e(\cdot,t) - v_\e^{\pm \infty} \|_{W^{1,2}(S^3)} \to 0$, as $t \to \pm \infty.$
\end{thm*}

The proof relies on the fact that the Clifford torus is the embedded, non-totally geodesic, minimal surface of least area $2 \pi^2$ in $S^3$, which was proved by Marques-Neves in their resolution of the Willmore conjecture \cite{marques2014min}. The main obstacle to extending these results to higher dimensions was the Solomon--Yau conjecture, which states that a minimal Clifford hypersurface has least area among all non-totally geodesic closed minimal hypersurfaces of $S^{n+1}$. This conjecture was recently resolved by OpenAI \cite{OAI:The-Solomon-Yau-least-volume-theorem-September-23-2026} for $n \geq 3$.

Here we give the expression of the minimal Clifford hypersurfaces:
\begin{align*}
T_k = S^k (\sqrt{\frac{k}{n}}) \times S^{n-k} (\sqrt{\frac{n-k}{n}}) \subset S^{n+1},
\end{align*}
where $n \geq 2$ and $1 \leq k \leq n-1$. $T_k$ has Morse index $n+3$ (see \cite{perdomo2001low}) and nullity $(k+1)(n-k+1)$ (see \cite{hsiang1971minimal}). Let $T_{min}^n$ denote the Clifford hypersurface in $S^{n+1}$ with lowest area among all Clifford hypersurfaces $T_k$ in $S^{n+1}$. The explicit expression of $T_{min}^n$ was given by the author in Appendix B, Proposition $3$ of \cite{Chen24}, as $T_{min}^n = T_p$, where $p = \left\lfloor \frac{n}{2}\right\rfloor$.

In \cite{Chen24}, the author also constructed an $S^1$ family of $\mathrm{SO}(n)$ invariant eternal weak mean curvature flows in $S^{n+1}$ connecting the Clifford hypersurface $T_1$ to equatorial spheres.

Combining the main theorem with Proposition $3.5$, Theorem $4.6$, Proposition $5.1$, and Corollary $6.6$ of \cite{OAI:The-Solomon-Yau-least-volume-theorem-September-23-2026}, we have the following classification and rigidity result:

\begin{thm}\label{classification}
Let $n\geq 2$. Among smooth non-totally geodesic minimal hypersurfaces in $S^{n+1}$, the least area is attained by $T_{\min}^n$. Moreover, $\omega_{n+3}(S^{n+1})=\text{Area}(T_{\min}^n).$ Every smooth minimal hypersurface $\Sigma\subset S^{n+1}$ satisfying $\text{Area}(\Sigma)=\text{Area}(T_{\min}^n)$ is congruent to $T_{\min}^n$.
\end{thm}

Here $\omega_k$ denotes the min-max width, see e.g. \cite{guthminmax,MarquesNevesRicci}.

\begin{rmk}
This computation provides a family of explicit examples in all dimensions whose second distinct min-max value is realized by a smooth minimal hypersurface, although general min-max theory allows singularities when $n\geq7$, see e.g. \cite{Zhou17}.
\end{rmk}

The terminology used here will be introduced in Section \ref{preliminary}. Equation \eqref{PAC} is the negative gradient flow of the Allen--Cahn energy $E_\e$. Let $u_\e^{-\infty}$ be a critical point of $E_\e$ in $S^{n+1}$ which is the Allen-Cahn approximation of the Clifford hypersurface $\mathcal{T} = \{x \in S^{n+1} \mid \sum\limits_{i=1}^{p+1} x_i^2 = \frac{p}{n}\}$. The existence of such solutions comes from \cite{CajuGaspar}. Recall from \cite{CGGM} that ground state solutions of the Allen-Cahn equation on $S^{n+1}$ accumulate on an equatorial sphere $S^n$. 

In this short note, using Theorem \ref{classification}, we generalize our previous theorems to all dimensions $n \geq 2$ and prove:

\begin{thm} \label{Morse}
For sufficiently small $\e>0$, for any ground state solution $u_\e^{+\infty}$ of the Allen-Cahn equation, there is an eternal solution $u_\e \colon S^{n+1} \times \R \to \R$ to \eqref{PAC} such that
\[\| u_\e(\cdot,t) - u_\e^{\pm \infty} \|_{W^{1,2}(S^{n+1})} \to 0, \quad \text{as} \quad t \to \pm \infty. \]
\end{thm}

\begin{thm}\label{flow}
There exist (infinitely many) eternal weak mean curvature flows in $S^{n+1}$ connecting $\mathcal{T}$ to equatorial spheres $S^n$. These flows are smooth for large $|t|$.
\end{thm}

Note that for $n\geq4$, the family of Clifford hypersurfaces considered here is different from that in the author's previous paper \cite{Chen24}. For $n=3$, that paper gives an $S^1$ family of flows, while here we use a smaller symmetry group to obtain a larger $S^2$ family.

\subsection*{Outline of the proof}

We first extend Theorem \ref{classification} to low-area stationary integral varifolds which are cyclic mod $2$, using dimension reduction. Together with the convergence and rigidity results recalled below, this gives the low-energy classification for Allen--Cahn solutions. We then adapt the Morse perturbation and topological arguments of \cite{ChenGaspar25} to prove Theorem \ref{Morse}, and the convergence and symmetry arguments of \cite{ChenGaspar23,Chen24} to prove Theorem \ref{flow}. We explain the modifications needed to adapt the previous proofs to higher dimensions.

\subsection*{Acknowledgment}
The author acknowledges the support of University of Macau Research Grant MYRG-SRG2026-00039-FSC. The author thanks P. Gaspar for helpful comments on the early draft.

\section{Preliminaries and low-energy classification} \label{preliminary}

\subsection{Allen--Cahn equations and varifold convergence} \label{ACminimal}

\begin{defi}
Let $(M^{n+1},g)$ be a closed Riemannian manifold. We define the \emph{Allen-Cahn energy} on $M$ by:
\[E_{\epsilon}(u):= \int_M \left(\frac{\epsilon}{2} |\nabla_{g} u|^2 + \frac{1}{\epsilon} W(u)\right) d\mu_{g}, \ \ u \in W^{1,2}(M),\]
where $d\mu_g$ is the volume measure with respect to $g$.
\end{defi}
Throughout, $W$ is a smooth symmetric double-well potential satisfying the assumptions in Section $2.1$ of \cite{ChenGaspar25}, including quadratic growth and bounded $W^{\prime\prime}$. For example, one may modify $W(t)=\frac14(1-t^2)^2$ outside $[-2,2]$ as described there.

One can check that $u$ is a critical point of $E_{\epsilon}$ on a closed manifold $(M^{n+1},g)$ if and only if $u$ (weakly) solves the \emph{elliptic Allen-Cahn equation}:
\begin{equation} \label{AC}
\epsilon^2 \Delta_g u -  W'(u)=0 \quad \text{on} \ M. \tag{AC}
\end{equation}

We write $\sigma = \int_{-1}^1 \sqrt{W(t)/2}\,dt$, so that $2\sigma$ is the one-dimensional transition energy.

For the quadratic form given by the second variation of the energy $E_\e$ at $u$, we can define the \emph{Morse index} and the \emph{nullity} of a solution $u$ of \eqref{AC} (as a critical point of $E_\e$), denoted  $\ind_\e(u)$ and $\text{nul}_{\e}(u)$, as the number of negative eigenvalues and the dimension of the kernel of the linear operator $\mathcal{L}_{\e,u}(f) = \Delta f - \frac{W''(u)}{\e^2}f,$ counted with multiplicity, respectively. 

We note that given $\e>0$ and a sufficiently regular function $u$ on $M$ (so that almost every level set is a regular hypersurface), we can consider the \emph{associated $n$-varifolds} $V_{\e,u}$ defined by
    \begin{equation*}
        V_{\e,u}(\phi) = \frac{1}{2}\int_{M\cap \{\nabla u \neq 0\}} \phi(x,T_x\{u=u(x)\})\cdot\left(\frac{\e|\nabla u(x)|^2}{2}+\frac{W(u(x))}{\e}\right)\,d\mu_g(x)
    \end{equation*}
for any continuous function $\phi$ on the Grassmann bundle $G_n(TM)$ of $n$-planes in the tangent spaces of $M$. We write $\mu_{\e,u}=\|V_{\e,u}\|$ for the associated Radon measure on $M$ (the \emph{weight measure} of $V_{\e,u}$). In the case where $\{u_j\}$ are solutions to \eqref{PAC}, with $\e = \e_j \downarrow 0$, we will write ${V_{\e_j,t}=V_{\e_j,u_j(\cdot,t)}}$ and ${\mu_{\e_j,t} := \mu_{\e_j,u_j(\cdot,t)}}$.

We have the following convergence results for solutions to \eqref{AC} and \eqref{PAC}.

\begin{thm*}[\cite{HT,TW,Guaraco}]
Let $(M^{n+1},g)$ be a closed Riemannian manifold. Let $\{u_j\}$ be a sequence of solutions of \eqref{AC} with $\e=\e_j \downarrow 0$. Suppose that $\sup_j E_{\e_j}(u_j) < \infty$. Then we can find a (not relabeled) subsequence of $u_j$ such that $V_{\e_j}$ converge to a stationary $n$-varifold $V$ on $M$ such that $\frac{1}{\sigma}V$ is integral. Moreover,
\[\frac{1}{\sigma}\|V\|(M) = \lim_{j \to \infty} \frac{1}{\sigma} \|V_{\e_j}\|(M) = \lim_{j \to \infty}\frac{1}{2\sigma}E_{\e_j}(u_{j}),\]
and $u_{j}$ converges uniformly to $\pm 1$ in compact subsets of $M \setminus \supp \|V\|$.

Furthermore, if $n\geq 2$ and if $\sup_j \ind_{\e_j}(u_j)<\infty$, then $\supp\|V\|$ is a smooth, embedded, minimal hypersurface in $M$ away from a closed set of Hausdorff dimension $\leq (n-7)$.
\end{thm*}

The minimal surface $\supp\|V\|$ is often called a \emph{limit interface} obtained from $u_j$. Solutions with limit interface $\Gamma$ are called the \emph{Allen-Cahn approximation} of $\Gamma$.

We state below the main convergence result we will use in the present article, which follows from \cite{ilmanen1993convergence} and the work of Tonegawa \cite{tonegawa2003integrality} (see also \cite{sato2008simple} and \cite{takasao2016existence}):

\begin{thm*}
Let $(M^{n+1},g)$ be a closed Riemannian manifold. Let $\{u_j\}$ be a sequence of solutions to \eqref{PAC} on $M \times [t_0,\infty)$ with $\e=\e_j \downarrow 0$. Suppose that there exist constants $c_0,E_0>0$ such that
\begin{enumerate}
    \item[(a)] $\sup_{M\times [t_0,\infty)}|u_j| \leq c_0$, for all $j$,
    \item[(b)] $E_{\e_j}(u_j(\cdot,t)) \leq E_0$, for all $t\geq t_0$ and all $j$, and
    \item[(c)] $\int_{M\times(t_0,\infty)}\e_j|\partial_t u_j|^2\,d\mu_g\,dt \leq E_0$, for all $j$.
\end{enumerate}
Then, passing to a subsequence (not relabeled), there are Radon measures $\{\mu_t\}_{t\geq t_0}$ such that
\begin{enumerate}
    \item[(i)] $\mu_{\e_j,t} \to \mu_t$  as Radon measures on $M$, and
    \[\frac{1}{2\sigma}\lim_{j \to\infty} E_{\e_j}(u_j(\cdot,t)) = \frac{1}{\sigma}\lim_{j \to \infty}\|\mu_{\e_j,t}\|(M) = \frac{1}{\sigma}\|\mu_t\|(M),\]
    for every $t>t_0$.
    \item[(ii)] For a.e. $t>t_0$, $\mu_t$ is $n$-rectifiable, and its density is $N(x)\sigma$, for $\mu_t$-a.e. $x \in M$, where $N(x)$ is a nonnegative integer.
    \item[(iii)] $\mu_t$ satisfies the mean curvature flow in the sense of Brakke, namely:
    \[\overline{D}_t \int_M \phi\,d\mu_t \leq \int_M (-\phi)\|H_t\|^2 + \langle\nabla \phi, H_t \rangle \, d\mu_t,\]
    for any $C^2$ function $\phi\geq 0$. Here $\overline{D}_t$ denotes the upper derivative, and $H_t$ is the generalized mean curvature vector of $\mu_t$.
\end{enumerate}
\end{thm*}

\subsection{Previous results} \label{previous result}

We recall the results and techniques from our joint papers \cite{ChenGaspar23,ChenGaspar25} with P. Gaspar that will be used below.

For a solution $u_{\e}^{-\infty}$ on $S^{n+1}$ to \eqref{AC} with Morse index $I$, by the result of Choi-Mantoulidis \cite{choimantoulidis22}, where they show the existence of ancient gradient flows and uniqueness under integrability conditions, and the theory of parabolic PDEs, we show that there is a family of gradient flows of $E_{\e}$,
\begin{align*}
\mathscr{S}_{\e}: B_{\eta} \subset \R^I \to C^2(S^{n+1} \times \R)
\end{align*}
such that $\lim_{t\to-\infty}\mathscr S_\epsilon(a)(\cdot,t)=u_\e^{-\infty}$ for every $a\in B_\eta$. The curve $t \to \mathscr{S}_{\e}(a)(\cdot, t)$ is tangent at $u_{\e}^{-\infty}$ to eigenfunctions $\{\varphi_j\}_{j=1}^I$ of 
\begin{align*}
D^2E_{\e}(u_{\e}^{-\infty}) = \e \Delta - \frac{1}{\e} W''(u_{\e}^{-\infty})
\end{align*}
with coefficients $a = (a_1, \cdots, a_I)$.

When $I \geq 2$, for small $r < \eta$, using the fact that the first eigenfunction $\varphi_1 > 0$, the maximum principle for parabolic equations, and the Frankel-type property proved by Hiesmayr \cite{hiesmayr2020rigidity}, we prove that $\mathscr{S}_{\e}(\pm r, 0, \cdots, 0)$ converge to $\pm 1$ as $t \to +\infty$.

As shown in \cite{ChenGaspar23}, the sets of initial data whose solutions converge to $1$ and $-1$ are open. By connectedness, every path in $B_\eta\setminus\{0\}$ joining $(r,0,\ldots,0)$ to $(-r,0,\ldots,0)$ contains a point $a$ such that $\mathscr{S}_\e(a)$ does not converge to either constant. For the solutions constructed on $S^3$ in \cite{ChenGaspar23}, the low-area classification and a parity argument identify the backward and forward limit interfaces as a Clifford torus and an equator, respectively, both for the Allen--Cahn flow and for its limiting Brakke flow. The uniqueness theorem for ancient gradient flows in \cite{choimantoulidis22} also gives the symmetry argument used in \cite{ChenGaspar23}: if an isometry preserving $u_\e^{-\infty}$ fixes the parameter $a$ in the unstable eigenspace, it preserves $\mathscr{S}_\e(a)(\cdot,t)$, its limiting Brakke flow, and their stationary limits. Together with a continuity argument, this yields a two-parameter family of Brakke flows on $S^3$.

In \cite{ChenGaspar25}, we construct Allen--Cahn flows on $S^3$ with prescribed limits using Morse theory. We perturb the energy near its low-energy critical manifolds, choosing auxiliary Morse functions with maxima at the prescribed ground states and minima at the prescribed Allen--Cahn approximations of the Clifford torus. A homological argument based on the fifth Allen--Cahn width produces connecting orbits for the perturbed flow. Letting the perturbation tend to zero and identifying the limiting endpoints then gives the desired connections.

\subsection{Low-area varifolds and Allen--Cahn rigidity}

We have the following density estimate for stationary integral varifolds.

\begin{lem}\label{density}
Let $V$ be an $n$-dimensional stationary integral varifold in $S^{n+1}$. Then
\[
\Theta^n(V,x)\leq\frac{\|V\|(S^{n+1})}{\text{Area}(S^n)}
\qquad\text{for every }x\in\supp\|V\|,
\]
where $\Theta^n(V,x)$ denotes the $n$-dimensional density of $V$ at $x$. In particular, if $\|V\|(S^{n+1})<2\text{Area}(S^n)$, then $V$ has multiplicity one.
\end{lem}

\begin{proof}
The volume of the $n+1$ dimensional ball with radius $r$ is $\text{Vol}(B_r^{n+1}) = \frac{r^{n+1}}{n+1} \text{Area}(S^n)$.

Let $C=0\#V$ be the stationary cone over $V$ in $\R^{n+2}$. Its radial product structure gives $\Theta^{n+1}(C,x)=\Theta^n(V,x)$ for $x\in\supp\|V\|\subset S^{n+1}$ and $\|C\|(B_R(0))=\frac{R^{n+1}}{n+1}\|V\|(S^{n+1})$.

Fix $x\in\supp\|V\|\subset S^{n+1}$. Since $|x|=1$, we have $B_R(x)\subset B_{R+1}(0)$. By the monotonicity formula, the mass ratio centered at $x$ is nondecreasing in $R$. Hence
\[
\Theta^{n+1}(C,x) \leq \lim_{R\to\infty}\frac{(n+1)\|C\|(B_R(x))}{R^{n+1}\text{Area}(S^n)} \leq \lim_{R\to\infty}\frac{(n+1)\|C\|(B_{R+1}(0))}{R^{n+1}\text{Area}(S^n)} =\frac{\|V\|(S^{n+1})}{\text{Area}(S^n)}.
\]

The last assertion follows from integrality, since under the stated area bound $\Theta^n(V,x)$ is a positive integer strictly less than $2$ for $\|V\|$-almost every $x$.
\end{proof}

We extend Theorem \ref{classification} to stationary integral varifolds by dimension reduction.

\begin{lem}\label{minimal}
Let $T$ be an $n$-dimensional stationary integral varifold in $S^{n+1}$. If $\text{Area}(T) \leq \text{Area}(T_{min}^n)$ and its associated $\Z_2$ chain $[T]$ has $\partial [T] = 0$, then $T$ is a multiplicity one equatorial sphere or a Clifford hypersurface $T_{min}^n$.
\end{lem}

\begin{proof}

We use an inductive argument. The case for $n = 2$ has been proved by Choi-Mantoulidis in Lemma $5.8$ of \cite{choimantoulidis22}. Now we prove the lemma for a $(n+1)-$dimensional stationary varifold $T$ in $S^{n+2}$ satisfying all conditions specified in the lemma, assuming that the lemma holds for the $n$-dimensional case.

If $T$ is smooth, Lemma \ref{density} and the area bound imply that it has one component with multiplicity one. The result then follows from Theorem \ref{classification}.

If $T$ were singular, then the $(n+2)-$dimensional stationary cone $C^{n+2}: = 0 \# T$ would be one where the origin is not an isolated singularity. Let $x_0 \neq 0$ be a singular point of $C^{n+2}$. By the proof of Lemma \ref{density} and Theorem $3$ of \cite{Chen24}, $\Theta^{n+2}(C^{n+2},x_0) \leq \frac{\text{Area}(T_{min}^{n+1})}{\text{Area}(S^{n+1})} < \frac{\text{Area}(T_{min}^{n})}{\text{Area}(S^{n})} < 2$.

Let $\overline{C}^{n+2}$ be a tangent cone to $C^{n+2}$ at $x_0$. Automatically, $\overline{C}^{n+2} \cong \overline{C}^{n+1} \times \R$ for some stationary $n+1$ dimensional cone $\overline{C}^{n+1} \subset \R^{n+2}$. Let $T^{n} \subset S^{n+1}$ be the link of $\overline{C}^{n+1}$. By construction, it is easy to see that $\partial[T^n]=0$. It has
\[\text{Area}(T^{n}) = \text{Area}(S^n) \Theta^{n+1}(\overline{C}^{n+1},0) = \text{Area}(S^n) \Theta^{n+2}(C^{n+2},x_0) < \text{Area}(T_{min}^{n}).\]

Thus $T^n$ satisfies all the assumptions of the induction hypothesis. By the $n-$dimensional case, we know that $T^n \simeq S^n$ with multiplicity one, so $\overline{C}^{n+1} \simeq \R^{n+1}$ with multiplicity one, and $\overline{C}^{n+2} \simeq \R^{n+2}$ with multiplicity one. This violates the singular nature of $x_0 \in C^{n+2}$ by Allard's theorem \cite{simon1983lectures}.

\end{proof}

By \cite{GG,Dey}, the Allen--Cahn widths $\omega_k(E_\e)$ satisfy $\lim\limits_{\e\downarrow0}\frac{\omega_k(E_\e)}{2\sigma} =\omega_k(S^{n+1}).$ We have the following classification result for low energy solutions.

\begin{thm} \label{lowenergy}
For sufficiently small $\epsilon$, let $u$ be a nonconstant solution of \eqref{AC} with energy $E_\e(u) \leq \omega_{n+3}(E_\e)$. Then either
\begin{itemize}
    \item[(i)] If $E_\e(u)<\omega_{n+3}(E_\e)$, then $u$ is a radially symmetric solution of Morse index 1, energy $E_\e(u) = \omega_1(E_\e) = \ldots = \omega_{n+2}(E_\e)$ and vanishes precisely on an equatorial sphere.
    \item[(ii)] If $E_\e(u)=\omega_{n+3}(E_\e)$, then up to an isometry $u$ is a $\mathrm{SO}(p+1)\times \mathrm{SO}(n-p+1)$-symmetric Allen--Cahn approximation of $T_{min}^n$ with Morse index $n+3$. 
\end{itemize}
\end{thm}

\begin{proof}
Let $\e_j\downarrow0$ and let $u_j$ be nonconstant solutions with $E_{\e_j}(u_j)\leq\omega_{n+3}(E_{\e_j})$. Passing to a subsequence, the varifolds $\frac{1}{\sigma}V_{\e_j,u_j}$ converge to a stationary integral varifold $V$ with area in $[\text{Area}(S^n),\text{Area}(T_{min}^n)]$. Lemma \ref{density} and the area bound imply that $V$ has multiplicity one. Then the parity argument in the proof of Lemma $1$ of \cite{ChenGaspar23} shows that $V$ is cyclic mod $2$. Lemma \ref{minimal} then implies that $V$ is an equator or a Clifford hypersurface $T_{min}^n$. The index estimates of \cite{gaspar2020second} and Theorem $5.11$ of \cite{ChodoshMantoulidis20}, together with the rigidity results of \cite{CGGM} and \cite{hiesmayr2020rigidity}, imply that, for large $j$, $u_j$ is a ground state solution or an Allen--Cahn approximation of the Clifford hypersurface $T_{min}^n$, with Morse index $1$ or $n+3$, respectively.

By convergence of the widths, critical points realizing $\omega_{n+3}(E_\e)$ cannot be ground states for small $\e$. The uniqueness up to isometry and change of sign proved in Theorem $4$ of \cite{hiesmayr2020rigidity}, together with the description of the first $n+2$ widths in \cite{CGGM}, then gives the stated energy levels.

\end{proof}

The solutions of the first type are called \emph{ground state solutions}, as they have least energy among nonconstant solutions of \eqref{AC}. Ground states and Allen--Cahn approximations of Clifford hypersurfaces with equal factor dimensions can be constructed by direct minimization and reflection \cite{CGGM}. For unequal factor dimensions, we use the construction of \cite{CajuGaspar}; see also Appendix A of \cite{hiesmayr2020rigidity}.

\section{Proofs of the main theorems} \label{proof of theorems}

\subsection{Proof of Theorem \ref{Morse}}

The proof follows Sections $3$ and $4$ of \cite{ChenGaspar25}. We first record the higher-dimensional version of Proposition $1$ of \cite{ChenGaspar25}.

\begin{prop}\label{energy-mb}
For sufficiently small $\epsilon$, the energy functional $E_\e$ is a Morse-Bott functional on $\{E_\e\leq\omega_{n+3}(E_\e)\}$. Its critical set in this sublevel is the disjoint union of
\begin{enumerate}
    \item the constant functions $1$ and $-1$;

    \item the set $ A_\e = \{E_\e=\omega_1(E_\e)=\cdots=\omega_{n+2}(E_\e)\} \cap\crit E_\e$ of ground state solutions, which is a smoothly embedded $(n+1)$-sphere in $W^{1,2}(S^{n+1})$.

    \item the set $B_\e = \{E_\e=\omega_{n+3}(E_\e)\}\cap\crit E_\e$ of Allen--Cahn approximations of $T_{min}^n$, symmetric under $\mathrm{SO}(p+1)\times\mathrm{SO}(n-p+1)$ up to rotation,  which is smoothly embedded in $W^{1,2}(S^{n+1})$ and is diffeomorphic to
    \[\begin{cases}
    G_{p+1}(\R^{n+2}), & n \text{ even},\\
    G_{p+1}(\R^{n+2})\sqcup G_{p+1}(\R^{n+2}), & n \text{ odd}.
    \end{cases}\]
\end{enumerate}
\end{prop}

The Morse--Bott property follows from the equality between the nullity and the dimension of the rotation orbit, as in Section $2.3$ and Appendix A of \cite{hiesmayr2020rigidity}.

The ground state orbit is parametrized by its pole and is therefore diffeomorphic to $S^{n+1}$. The rotation orbit of an Allen--Cahn approximation of the Clifford hypersurface $T_{min}^n$ is parametrized by the $(p+1)$-dimensional subspaces of $\R^{n+2}$ and is diffeomorphic to $G_{p+1}(\R^{n+2})$. When $n$ is even, a rotation interchanging the two sphere factors sends the solution to its negative, so they belong to the same orbit. When $n$ is odd, the two sphere factors have different dimensions, thus the two signs give distinct orbits. In either case, $B_\e/\{u\sim -u\}$ is connected, so the auxiliary even Morse functions can be chosen as in Section $3$ of \cite{ChenGaspar25}.

We retain $H=W^{1,2}(S^{n+1})$ as the variational space. The quadratic growth condition, boundedness of $W''$, and compactness of the embedding $H\hookrightarrow L^2$ give the $C^2$ regularity and Palais--Smale condition in every dimension. The Morse-theoretic argument only requires the $C^2$ Morse lemma and its homological consequences; see \cite{DucHungKhai07}.

For the parabolic argument, we retain $X=L^2(S^{n+1})$ and $X^\alpha=W^{2\alpha,2}(S^{n+1})$, with $3/4<\alpha<1$, as in Section $2.3$ of \cite{ChenGaspar25}. The inclusion $X^\alpha\hookrightarrow H$ holds in every dimension. Although the embedding into $L^\infty$ need not hold in higher dimensions, boundedness of $W''$ makes $W'$ globally Lipschitz on $X$. By Lemma $6$ of \cite{ChenGaspar25}, the coefficient of $\Delta u$ is $1+\delta\theta(u)$, where $\delta$ is the perturbation parameter and $\theta:H\to\R$ is bounded and Lipschitz, so this coefficient remains uniformly positive for sufficiently small $\delta$. The remaining perturbation terms are Lipschitz from $H$ to $X$, so the lower-order terms are locally Lipschitz from $X^\alpha$ to $X$. Since the perturbed functional differs from $\e^{-1}E_\e$ by a bounded term, its decrease gives a uniform $H$ bound along the flow for fixed $\e$. The lower-order terms are therefore bounded in $X$. The estimates in \cite{Amann88,Henry81} give a uniform $X^\beta$ bound for times bounded away from zero, where $\alpha<\beta<1$. The continuation criterion of \cite{Amann88}, recalled in Theorem $2$(iii) of \cite{ChenGaspar25}, then gives existence for all positive time.

With the auxiliary Morse functions chosen as in \cite{ChenGaspar25}, the solutions $\pm u_\e^{+\infty}$ and $\pm u_\e^{-\infty}$ have perturbed Morse indices $n+2$ and $n+3$, respectively. All critical points below the level of $\pm u_\e^{+\infty}$ have index at most $n+1$. Thus the homological argument of Section $4$ of \cite{ChenGaspar25} applies with unstable disks $D^{n+3}$, boundaries $S^{n+2}$, and $(n+3)$-sweepouts. We then let the perturbation tend to zero and identify the limits as in Section $3$ of \cite{ChenGaspar25}.

\subsection{Proof of Theorem \ref{flow}}

We follow the arguments in Sections $4$ and $5$ of \cite{ChenGaspar23}, replacing their low-energy classifications by Theorem \ref{lowenergy} and Lemma \ref{minimal}.

Let $G_1=\mathrm{SO}(p+1)$ and $G_2=\mathrm{SO}(n-p+1)$ act on the first $p+1$ and last $n-p+1$ coordinates of $S^{n+1} \subset \R^{n+2}$, respectively. By Theorem \ref{lowenergy}(ii), after a rotation we may assume that $u_\e^{-\infty}$ is $G_1\times G_2$ invariant and has Morse index $n+3$. It depends only on the lengths of the two coordinate vectors, so it is also invariant under coordinate reflections.

Let $\nu$ be a unit normal to $\mathcal T$. By equation $(2.1)$ and Theorem $2.4$ of \cite{perdomo2001low}, the constant and coordinate functions are Jacobi eigenfunctions with eigenvalues $-2n$ and $-n$, respectively. Thus $\nu,x_1\nu,\ldots,x_{n+2}\nu$ span an $(n+3)$-dimensional subspace on which the second variation of area is negative definite. Extend these normal vector fields by parallel transport along normal geodesics, with a cutoff depending only on distance to $\mathcal T$ and equal to $1$ near it. By the proof of Theorem $A$ of \cite{gaspar2020second}, using equation $(7)$, differentiating $u_\e^{-\infty}$ along these extended fields gives functions spanning an $(n+3)$-dimensional subspace on which the second variation of $E_\e$ at $u_\e^{-\infty}$ is negative definite for sufficiently small $\e$. This construction preserves the rotation and reflection actions on the original normal variations.

The $L^2$ projection of this subspace onto the unstable eigenspace is injective: a function in this subspace with zero projection has nonnegative second variation and must therefore be zero. Both spaces have dimension $n+3$, so the projection is an isomorphism. It preserves rotations and reflections because these isometries preserve the linearized operator and the $L^2$ inner product. The projected function arising from the constant normal variation is even in every coordinate, while the one arising from $x_i\nu$ is odd in the $i$-th coordinate and even in all others. Since these projected functions form a basis of the unstable eigenspace, each of these specified parities defines a one-dimensional subspace within it. The linearized operator preserves these subspaces, so the projected functions are eigenfunctions.

The positive first eigenfunction is fixed by these reflections, so it spans the subspace of functions even in every coordinate. We may therefore choose an $L^2$-orthonormal basis $\{\varphi_i\}_{i=1}^{n+3}$, with the first eigenfunction $\varphi_1>0$, so that $G_1$ acts on $\{\varphi_2,\ldots,\varphi_{p+2}\}$ by the same rotations as on the first $p+1$ coordinates and fixes the remaining eigenfunctions, while $G_2$ acts similarly on $\{\varphi_{p+3},\ldots,\varphi_{n+3}\}$. The proof of Lemma $6$ of \cite{ChenGaspar23} then gives invariance of $\mathscr S_\e(a)$ whenever its parameter is fixed by an isometry preserving $u_\e^{-\infty}$.

When $n=2p$, interchanging the sphere factors sends $u_\e^{-\infty}$ to its negative and preserves the linearized operator since $W$ is even. As in Section $5.1$ of \cite{ChenGaspar23}, choose $\varphi_2$ and $\varphi_{p+3}$ so that they are exchanged by this isometry. Take small equal and opposite coefficients in these directions, with all others zero. The resulting solution changes sign under the interchange, as in equation $(8)$ of \cite{ChenGaspar23}, and therefore cannot converge to $\pm1$. It is invariant under rotations and reflections fixing the two chosen directions, hence under $\mathrm{SO}(p)\times\mathrm{SO}(p)$.

When $n=2p+1$, use a small semicircle in the $(\varphi_1,\varphi_{p+3})$ coefficient plane joining $(r,0,\ldots,0)$ to $(-r,0,\ldots,0)$, as in the proof of Proposition $1$ of \cite{ChenGaspar23}. The same connectedness argument gives a solution which does not converge to $\pm1$. It is $G_1$ invariant and is also invariant under rotations and reflections in the second coordinate group fixing the chosen direction.

The proof of Proposition $2$ of \cite{ChenGaspar23}, using Theorem \ref{lowenergy}, gives ground state forward limits. Translate time so that $E_\e(u_\e(\cdot,0))=\sigma(\text{Area}(S^n)+\text{Area}(\mathcal{T}))$. The convergence and parity arguments in Section $4.2$ of \cite{ChenGaspar23}, using Lemmas \ref{density} and \ref{minimal}, then give an integral Brakke flow with Clifford backward limits and equatorial forward limits.

The varifolds inherit these symmetries, leaving finitely many compatible limits: for $n\geq3$, two least-area Clifford hypersurfaces, and two equators when $n$ is even or one when $n$ is odd. The case $n=2$ is treated in Sections $5.2$ and $5.3$ of \cite{ChenGaspar23}. After taking a subsequence with a fixed forward equator, Propositions $3$ and $4$ of \cite{ChenGaspar23} identify the backward limit as $\mathcal{T}$ and the forward limit as this equator. Brakke's Local Regularity Theorem \cite{ilmanen1994elliptic} gives smoothness for large $|t|$.

Finally, rotating the two chosen directions within their coordinate groups gives an $S^p\times S^p$ family when $n=2p$, while rotating the chosen direction in the second coordinate group gives a $G_1$ invariant $S^{p+1}$ family when $n=2p+1$. The same rotation argument as in Section $5.3$ of \cite{ChenGaspar23} gives continuous and equivariant dependence on these parameters.

\bibliographystyle{amsalpha}
\bibliography{main}

\end{document}